\documentclass[a4paper]{article}

\usepackage[utf8]{inputenc} 

\usepackage{amsmath,amssymb,amsthm}
\usepackage{enumitem,mathrsfs}
\usepackage{hyperref}

\setenumerate{label=\rm(\roman*),labelindent=.5em,widest=iv,leftmargin=*}
\setitemize{label=\textbullet,labelindent=.5em,leftmargin=*}

\usepackage[paper=A4,pagesize]{typearea}

\newcommand\R{\mathbb{R}}
\newcommand\N{\mathbb{N}}
\newcommand\osc{\mathrm{osc}}

\newtheorem{thm}{Theorem}
\newtheorem{lem}[thm]{Lemma}
\newtheorem{fct}[thm]{Fact}
\newtheorem{obs}[thm]{Observation}
\newtheorem{prp}[thm]{Proposition}
\newtheorem{cor}[thm]{Corollary}

\theoremstyle{definition}

\newtheorem{dfn}[thm]{Definition}

\newenvironment{cs}[1]{\begin{trivlist}\item[\hskip\labelsep{\bf\boldmath Case #1.}]}{\end{trivlist}}

\newcommand{\MC}[1]{\ifx\relax#1\relax\textup{(MC)}\else\textup{(MC$_{#1}$)}\fi}

\newcommand{\VB}{\ensuremath{\mathrm{VB}}}
\newcommand{\AC}{\ensuremath{\mathrm{AC}}}
\newcommand{\ACG}{\ensuremath{\mathrm{ACG}}}
\newcommand{\VBs}{\ensuremath{\mathrm{VB_\star}}}
\newcommand{\ACs}{\ensuremath{\mathrm{AC_\star}}}
\newcommand{\ACGs}{\ensuremath{\mathrm{ACG_\star}}}

\newcommand{\J}{\mathscr{R}}
\newcommand{\I}{\mathscr{C}}

\title{Monotonically controlled integrals\thanks{The research leading to these results has received funding from the European Research Council under the European Union's Seventh Framework Programme (FP/2007-2013) / ERC Grant Agreement n.2011-ADG-20110209}}
\author{Thomas Ball\thanks{Granary Wharf House, 2 Canal Wharf, Leeds, LS11 5PS; email: thomas.ball@lhasalimited.org}\, and David Preiss\thanks{Mathematics Institute, University of Warwick, Coventry CV4 7AL, United Kingdom; email d.preiss@warwick.ac.uk}}
\date{}

\begin{document} 

\maketitle

\begin{abstract}
The monotonically controlled integral defined
by  Bendov\'a and Mal\'y, which is equivalent to the Denjoy-Perron integral, 
admits a natural parameter $\alpha>0$ thereby leading to
the whole scale of integrals called $\alpha$-monotonically controlled integrals. 
While the power of these integrals is easily seen to increase with increasing $\alpha$,
our main results show that
their exact dependence on $\alpha$ is rather curious.
For $\alpha<1$ they do not even contain the Lebesgue integral, 
for $1\le\alpha\le 2$ they coincide with the Denjoy-Perron integral, and for $\alpha>2$ 
they are mutually different and not even contained in the Denjoy-Khintchine integral.
\end{abstract}

\section{Introduction}
The monotonically controlled integral, or \MC{} integral, defined
by Hana Bendov\'a and Jan Mal\'y in \cite{BM},
is an interesting variant of nowadays rather abundant equivalent
definitions of the Denjoy-Perron integral. It is defined as follows.

\begin{dfn} Let $I\subset \R$ be an open interval and $f,F:I\to\R$ be functions.
We say that $f$ is an \textit{\MC{} derivative}
(monotonically controlled derivative)
of $F$  on $I$, or  that $F$ is an \textit{indefinite \MC{} integral}
of $f$ on $I$, if there exists a strictly increasing
function  $\varphi: I\to \R$ (which is called \textit{control function} for
the pair $(F,f)$)
such that for each $x\in I$,
\begin{equation}\label{EBM}
\lim_{y \to x}{\frac{F(y)-F(x)-f(x)(y-x)}{\varphi(y)-\varphi(x)}}=0.
\end{equation}
\end{dfn}

As Jan Mal\'y pointed out,
this definition invites introduction of a natural parameter thereby leading to
the whole scale of $\alpha$-monotonically controlled integrals, or \MC{\alpha} integrals,
where $\alpha>0$ is a parameter.

\begin{dfn} 
Let $I\subset \R$ be an open interval, $f,F:I\to\R$ be functions and $\alpha>0$. 
We say that $f$ is an \textit{\MC{\alpha} derivative} of $F$ on $I$,
or  that $F$ is an \textit{indefinite \MC{\alpha} integral}
of $f$ on $I$, if there exists a strictly increasing
function $\varphi: I \to\R$ (which we call the $\alpha$-control function for the pair $(F, f )$) such that for
each $x\in I$,
\[
\lim_{y \to x}{\frac{F(y)-F(x)-f(x)(y-x)}{\varphi(x+\alpha(y-x))-\varphi(x)}}=0.
\]
\end{dfn}

The original \MC{} integral is obtained for $\alpha=1$
and, since \MC{\alpha} integrability implies \MC{\beta} integrability for $\beta>\alpha$,
one of a number of natural questions, asked by Mal\'y, is whether for 
$\alpha>1$ the \MC{\alpha} integral
is still equivalent to the Denjoy-Perron integral.  
Here we show that this is not the case, but that the situation is rather interesting:
while for $\alpha\in[1,2]$ the \MC{\alpha} integral is equivalent to the  Denjoy-Perron integral,
for $\alpha>2$ the \MC{\alpha} integrals are mutually different and
even are not contained in the Denjoy-Khintchine integral. 
Although the case $\alpha<1$ is perhaps less interesting, we at least show that
these integrals do not contain the Lebesgue integral.
More precisely, our main results imply the following
\begin{thm}\label{MC} 
On any open interval $I\subset\R$,
\begin{enumerate}
\item\label{MC.1} 
There is a Lebesgue integrable function that is not
\MC{\alpha} integrable for any $0<\alpha<1$.
\item\label{MC.2} 
For $1\le\alpha\le2$, the \MC{\alpha} integral coincides with the Denjoy-Perron integral.
\item\label{MC.3} 
For any $\alpha\ge 2$ there is a function which is not \MC{\alpha} integrable but is \MC{\beta} integrable 
for every $\beta>\alpha$.
\item\label{MC.4} 
There is a function that is \MC{\alpha} integrable for every $\alpha>2$, but is not 
Denjoy-Khintchine integrable.
\end{enumerate}
\end{thm}

In connection with \ref{MC.1} and \ref{MC.4}, standard examples show that
there is a function that is \MC{\alpha} integrable for any $\alpha<1$ but not 
Lebesgue integrable, and a function that is \MC{\alpha} integrable for any $\alpha> 2$ but not 
Denjoy-Khintchine integrable.

Theorem \ref{MC}, sometimes in a stronger or more precise form, will be proved
in the last part of this paper:
 \ref{MC.1} in Theorem~\ref{M.1},
 \ref{MC.2} in Theorem~\ref{M.2},
 \ref{MC.3} in Theorem~\ref{M.3}, and
 \ref{MC.4} in Theorem~\ref{M.4}.

Before coming to the above results, we show in Section~\ref{bs} some basic properties of these integrals,
especially that the notion is reasonable, namely that two indefinite integrals of
the same function differ by a constant.
However we do not develop more advanced theory of these integrals.

The main results are stated and proved in Section~\ref{rel}. We first
remind ourself with the notions used to give equivalent
definitions of the Lebesgue, Denjoy-Perron and Denjoy-Khintchine 
integrals that we use to compare them with the \MC{\alpha} integral.
In addition to
results stated in Theorem~\ref{MC} we 
also show that bounded measurable functions
are \MC{\alpha} integrable for every $\alpha>0$ 
and that every function that is \MC{\alpha} integrable for some $\alpha>0$
is Lebesgue integrable on some subinterval. 
In connection with the statement~\ref{MC.2} of Theorem~\ref{MC}
we reprove the result of Bendov\'a and Mal\'y~\cite{BM} that the \MC{} integral
is equivalent to the Denjoy-Perron integral by showing directly 
its equivalence with Perron's original definition \cite{OP}. The argument in \cite{BM}
is based on the definition of the Kurzweil-Henstock integral \cite{JK,RH},
which is known to be equivalent to the Denjoy-Perron integral. However,
we should point out that our approach is close to Kurzweil's proof of equivalence
of his integral (which later became the  Kurzweil-Henstock integral) to the Perron integral.
The closeness of the definition of the \MC{} integral to Perron's definition is perhaps
surprising, since \cite{BM} quotes as the intermediate step to their definition the 
variational integral of \cite{RG} which arose from Henstock's approach~\cite{RH}.

In this area, 
it is no wonder that most of our references are to the still best text on much of 
classical real analysis, 
\emph{Theory of the Integral} by Stanis{\l}aw Saks. As this book numbers paragraphs 
and results in each chapter separately, we will refer, for example, to VII(\S5) for 
paragraph 5 in chapter 7, or to Theorem VI(7.2) for Theorem (7.2) in chapter 6.

As in \cite{BM}, although this makes little difference,
we are interested in indefinite integrals rather than in definite ones.
In particular, by saying that a function is \emph{integrable} we mean
that it has an indefinite integral.  Finally, we mention that 
we will use the terms \emph{positive} for ${}\ge0$, \emph{strictly positive}
for $>0$, and similarly
\emph{increasing} and \emph{strictly increasing}.

\section{Basic properties of \MC{\alpha} integrals}\label{bs}

We begin by remarking that it is immediate to see that if $\varphi$ is an $\alpha$-control function for
$(F,f)$, $c>0$ and $\psi$ is increasing, then $c\varphi+\psi$ is also an $\alpha$-control function for
$(F,f)$. It follows that, if $\varphi$ is an $\alpha$-control function for
$(F,f)$, $\psi$ is an $\alpha$-control function for
$(G,g)$ and $a,b\in\R$, then $\varphi+\psi$ is a control function for
$(aF+bG,af+bg)$. In other words, the \MC{\alpha} integral is linear:
If $F$ and $G$ are indefinite \MC{\alpha} integrals of $f$ and~$g$, respectively,
and $a,b\in\R$, then $aF+bG$ is an indefinite \MC{\alpha} integral of $af+bg$.

Some simple basic properties of
indefinite \MC{\alpha} integrals are collected in the following statement.

\begin{prp}\label{T}
Suppose $\alpha>0$ and $F$ is an indefinite \MC{\alpha} integral of $f$ on $(a,b)$. Then
\begin{enumerate}
\item\label{T.1}
$F$ is continuous on $(a,b)$;
\item\label{T.2}
$F$ is an indefinite \MC{\beta} integral of $f$ on $(a,b)$ for every $\alpha>\beta$;
\item\label{T.3}
$F'(x)=f(x)$ for almost every $x\in(a,b)$;
\end{enumerate}
\end{prp}

\begin{proof}
Let $\varphi$ be an $\alpha$-control function for the pair $(F,f)$.
Since $\varphi$ is bounded on a neighbourhood of $x$, taking limit as $y\to 0$ in
\[F(y) = F(x) + f(x)(y-x)+  \frac{F(y)-F(x)-f(x)(y-x)}{\varphi(x+\alpha(y-x))-\varphi(x)}
(\varphi(x+\alpha(y-x))-\varphi(x)) \]
gives
\[\lim_{y \to x} F(y) =F(x).\]

The second statement follows immediately from 
\[\varphi(x+\alpha(y-x))-\varphi(x)\le \varphi(x+\beta(y-x))-\varphi(x).\]

For the third statement we observe that for 
every $x$ at which $\varphi$ is differentiable,
\begin{multline*}
\lim_{y \to x} \frac{F(y)-F(x)-f(x)(y-x)}{y-x}\\
=\alpha \lim_{y \to x}\frac{F(y)-F(x)-f(x)(y-x)}{\varphi(x+\alpha(y-x))-\varphi(x)}
\lim_{y \to x} \frac{\varphi(x+\alpha(y-x))-\varphi(x)}{\alpha(y-x)}
=0.
\end{multline*}
Since $\varphi$ is differentiable almost everywhere, \ref{T.3} follows.
\end{proof}

\begin{prp}\label{sm}
Suppose $F:(a,b)\to\R$ satisfies
\begin{enumerate}
\item\label{sm.1}
$\limsup_{y\nearrow x}F(y)\le F(x)\le\limsup_{y\searrow x} F(y)$ for every $x\in (a,b)$;
\item\label{sm.2}
there are a 
strictly increasing $\varphi:(a,b)\to\R$ and $\alpha>0$ such that
for every $x\in(a,b)$ except at most countably many,
\begin{equation}\label{E.sm}
\liminf_{h\searrow 0} \frac{F(x+h)-F(x)}{\varphi(x+\alpha h) -\varphi(x)}\ge 0.
\end{equation}
\end{enumerate}
Then $F$ is increasing on $(a,b)$.
\end{prp}

\begin{proof}
Fix $\tau>0$ and notice that, since $\varphi$ is increasing,
$F_\tau:=F+\tau\varphi$ also satisfies \ref{sm.1}. We show that
for every $x\in (a,b)$ at which \eqref{E.sm} holds
\begin{equation}\label{E.m}
\limsup_{h\searrow 0} \frac{F_\tau(x+h)-F_\tau(x)}{h}\ge 0.
\end{equation}
Then \cite[Theorem VI(7.2)]{Saks} says that $F_\tau$ is increasing 
and the statement will follow by taking the limit as $\tau\to0$.

Fix $x\in (a,b)$ at which \eqref{E.sm} holds.
Suppose first that there is a sequence $h_k\searrow 0$ such that 
\[L:=\lim_{k\to\infty}\frac{\varphi(x+\alpha h_k)-\varphi(x)}{h_k}.\]
exists and is finite. Since the limit defining $L$ consists of positive terms,
\[\limsup_{h\searrow 0} \frac{F_\tau(x+h)-F_\tau(x)}{h} 
\ge L \liminf_{k\to\infty} \frac{F(x+h_k)-F(x)}{\varphi(x+\alpha h_k) -\varphi(x)} \ge 0.\]

Suppose next that there is no $L$ as above or, in other words,  that $\varphi$ has infinite right derivative at $x$. 
Let $u_k:=\sup\{h\in[0,\delta]: \varphi(x+h)\ge \varphi(x)+k h\}$.
Using that $\varphi$ has infinite right derivative at $x$ and is bounded by $\varphi(x+\delta)$,
we see that $0<u_k \le (\varphi(x+\delta)-\varphi(x))/k$. Since $\varphi$ is increasing,
$\varphi(x+u_k)\ge \varphi(x)+k u_k$. For sufficiently large $k$
we have $\alpha u_k\in [0,\delta]$ and $\alpha u_k>u_k$,
hence $\varphi(x+\alpha u_k)-\varphi(x) \le k \alpha u_k \le (\varphi(x+u_k) - \varphi(x))/\alpha$
for large enough~$k$. 
Passing to a subsequence of $u_k$, which we will denote $h_k$, we have $h_k\searrow 0$ and
\[L:=\lim_{k\to\infty} \frac{\varphi(x+\alpha h_k)-\varphi(x)}{\varphi(x+ h_k) -\varphi(x)}\]
exists and is finite. Hence
\[\limsup_{h\searrow 0} \frac{F_\tau(x+h)-F_\tau(x)}{{\varphi(x+h)-\varphi(x)}} 
\ge \tau +L\liminf_{k\to\infty} \frac{F(x+h_k)-F(x)}{\varphi(x+\alpha h_k) -\varphi(x)}\ge\tau>0.\]  
Consequently, there are arbitrarily small $h>0$ such that $F_\tau(x+h)-F_\tau(x)>0$,
and so \eqref{E.m} holds as well.
Hence it holds in both cases, and the proof is finished.
\end{proof}

In order to see the connection with the monotonicity result of 
\cite[Theorem VI(7.2)]{Saks} on which its proof based,
Proposition~\ref{sm} is stated in considerably greater generality than we need
to prove the following Theorem.
However, we did not
attempt to find its strongest version. For example, our proof would allow the $\varphi$ to depend
on $x$, and, rather obviously, one
may replace the $\liminf$ in \ref{sm.2} by $\limsup$ for $\alpha\le 1$.

\begin{thm}\label{Mm}
Suppose $F$ is an indefinite \MC{\alpha} integral of $f\ge 0$ on $(a,b)$.
Then $F$ is increasing on $(a,b)$.
\end{thm}

\begin{proof}
This is immediate from Proposition~\ref{sm}, since its assumption \ref{sm.1} 
holds by continuity of $F$ and \ref{sm.2} holds with the $\alpha$-control function
for the pair $(F,f)$ by the definition of the
\MC{\alpha} integral. 
\end{proof}

\begin{cor}
For any $\alpha > 0$, an indefinite \MC{\alpha} integral of a function on an interval is unique up
to an additive constant.
\end{cor}

\begin{proof}
If $F,G$ are two indefinite
\MC{\alpha} integrals of $F$, by linearity both $F-G$ and $G-F$ are indefinite
\MC{\alpha} integrals of zero. Hence $F-G$ and $G-F$ are both increasing
by Theorem~\ref{Mm},
showing that they differ by a constant.
\end{proof}

\section{Relations between the \MC{\alpha}, Lebesgue, Denjoy-Perron 
and Denjoy-Khintchine integrals}\label{rel}

We begin by recalling notions related to the definitions or properties of the Lebesgue, Denjoy-Perron 
and Denjoy-Khintchine integrals that we use in our arguments. 
They all come from \cite{Saks},
where much more material on these integrals and notions may be found.

\begin{dfn}\label{VB}
A real-valued function $F$ defined on a set $E\subset\R$ is said to be 
\begin{itemize}
\item
of bounded variation (\VB) on $E$ if there is $v\in[0,\infty)$ such that
$\sum_i|F(b_i)-F(a_i)|\le\nobreak v$ 
for every sequence of non-overlapping intervals whose end-points belong to $E$,
\item
absolutely continuous (AC) on $E$ if for every $\varepsilon>0$ there is $\delta>0$ 
such that for every sequence of non-overlapping intervals whose end-points belong to $E$,
the inequality $\sum_i(b_i-a_i)<\delta$ implies $\sum_i|F(b_i)-F(a_i)|<\varepsilon$,
\item
generalized absolutely continuous (\ACG) on $E$ if it is continuous on $E$
and $E$ is the union of countably many sets on which $F$ is absolutely continuous,
\end{itemize}
\end{dfn}

\noindent For the next definition we recall that the oscillation of a function $F$ on a set is
\[\osc(F,E):=\sup_{x,y\in E} |F(x)-F(y)|.\]

\begin{dfn}
Assuming in Definition~\ref{VB} that $F$ is defined on an interval containing $E$, and replacing
$\sum_i|F(b_i)-F(a_i)|$ by $\sum_i\osc(F,[a_i,b_i])$ we obtain the notion
of functions 
\begin{itemize}
\item
of bounded variation in the restricted sense (\VBs) on $E$,
\item
absolutely continuous in the restricted sense (\ACs) on $E$,
\item
generalized absolutely continuous in the restricted sense (\ACGs) on $E$,
\end{itemize}
respectively.
\end{dfn}

Although we do not need it, we point out the important results 
that \ACGs\ functions are
differentiable almost everywhere and \ACG\ functions are approximately differentiable 
almost everywhere. For the definition of approximate derivative see \cite[VII(\S3)]{Saks}; here
we just need to know that if an ordinary derivative $F'(x)$ exists, then so does
the approximate derivative $F'_{\mathrm{ap}}(x)$ and $F'_{\mathrm{ap}}(x)=F'(x)$.

Finally, we collect so called descriptive definitions of the three integrals that we use. 
Again, \cite{Saks} is a reference for (often deep) proofs of equivalence
to more usual definitions.

\begin{fct}
Suppose $F$ and $f$ are real-valued functions on an interval $(a,b)\subset\R$.
\begin{itemize}
\item
$F$ is an indefinite Lebesgue integral of $f$ on $(a,b)$ if 
it is \AC\ (or equivalently \ACs) on $(a,b)$ and $F'(x)=f(x)$ for almost all $x\in(a,b)$,
\item
$F$ is an indefinite Denjoy-Perron integral of $f$ on $(a,b)$ if 
it is \ACGs\ on $(a,b)$ and $F'(x)=f(x)$ for almost all $x\in(a,b)$,
\item
$F$ is an indefinite Denjoy-Khintchine integral of $f$ on $(a,b)$ if 
it is \ACG\ on $(a,b)$ and $F'_{\mathrm{ap}}(x)=f(x)$ for almost all $x\in(a,b)$.
\end{itemize}
\end{fct}

\subsection{Relations between Lebesgue and \MC{\alpha} integrals}

From Proposition~\ref{T} and the coincidence of the \MC{} and Denjoy-Perron integral
we know that  \MC{\alpha} and Lebesgue integrals agree provided that they both exist.
Here we show that in general the ``both'' cannot be replaced by existence of just one of them.
In our first result we observe that the functions providing an example cannot be bounded.

\begin{prp}
Every measurable function that is locally bounded on $(a,b)$
has an indefinite \MC{\alpha} integral on $(a,b)$ for every $\alpha>0$.
\end{prp}

\begin{proof}
Let $F$ be an indefinite Lebesgue integral of a bounded measurable function $f$ on $(a,b)$.
Let $M:=\{x: F'(x)=f(x)\}$ and $N=(a,b)\setminus M$. Since $N$ has Lebesgue measure zero,
there are open sets $(a,b)\supset G_k\supset N$ such that $|G_k|<4^{-k}$.
Define $\varphi_k(x):=2^k|(a,x)\cap G_k|$ and $\varphi(x):=x+\sum_{k=1}^\infty \varphi_k(x)$. Clearly,
$\varphi$ is an increasing function on $(a,b)$.
We fix $\alpha>0$ and show that $\varphi$ is an $\alpha$-control function for 
the pair $(F,f)$. 

If $x\in M$, \eqref{EBM} follows immediately from $F'(x)=f(x)$, since then
\[
\limsup_{y \to x}\frac{|F(y)-F(x)-f(x)(y-x)|}{|\varphi(x+\alpha(y-x))-\varphi(x)|}
\le
\lim_{y \to x}\frac{|F(y)-F(x)-f(x)(y-x)|}{\alpha |y-x|}
=0.
\]

If $x\in N$, choose $c\in(0,\infty)$ such that $|f|\le c$ on a neighbourhood of $x$.
Then $|F(y)-F(x)-f(x)(y-x)|\le 2c|y-x|$ 
and, given any $k\in\N$, we see that 
\[\varphi(x+\alpha(y-x))-\varphi(x)=2^k\alpha(y-x)\]
for $y$ close enough to $x$. Hence
\[
\limsup_{y \to x}\frac{|F(y)-F(x)-f(x)(y-x)|}{|\varphi(y)-\varphi(x)|}
\le
\lim_{y \to x}\frac{2c|y-x|}{2^k|y-x|}
=2^{-k-1}c.
\]
Since $k$ can be arbitrarily large,  \eqref{EBM} follows.
\end{proof}

The first part of the above proof also shows that, if $F'(x)=f(x)$ at
every point $x\in(a,b)$, $F$ is an indefinite \MC{\alpha} integral of $f$.
Together with
well known examples of non-absolutely integrable derivatives we get
examples of functions that are \MC{\alpha} integrable for every $\alpha>0$
but not Lebesgue integrable. For completeness, we record this in the
following observation.

\begin{obs}
If $F$ is everywhere differentiable on $(a,b)$, then it is an indefinite \MC{\alpha} integral
of $F'$ for every $\alpha >0$. Hence there are functions that are \MC{\alpha} integrable 
for every $\alpha>0$ but not Lebesgue integrable.
\end{obs}

The following Lemma provides the main building block
in definition of a function that is Lebesgue but not \MC{\alpha} integrable for any $0<\alpha<1$.

\begin{lem}\label{LC2}
For any interval $J\subset\R$ and $0<\varepsilon,\tau<1$ 
there is a measurable function $f:\R\to[0,\infty)$ such that 
\begin{enumerate}
\item\label{LC2.i}
$\int_{-\infty}^\infty f(x)\,dx\le\varepsilon$;
\item\label{LC2.ab}
there is a finite collection of  
intervals $[a_i,b_i]\subset J$ such that $\sum_i\int_{a_i}^{b_i} f(x)\,dx>1/\varepsilon$
and the intervals $[a_i,a_i+\tau(b_i-a_i)]$
are mutually disjoint.
\end{enumerate}
\end{lem}

\begin{proof}
We choose $\sigma\in(\tau,1)$,
an interval $[a,b]\subset J$ and $m\in\N$ such that
$m>1/\varepsilon^2$. For $i=0,1,\dots,m$ we recursively define
$a_0:=a$ and $a_i:=a_{i-1}+\sigma(b-a_i)$ and
notice that $a=a_0<a_1<\dots<a_{m}<b$.
Let
\[ f(x):= 
\begin{cases} \varepsilon/(b-a_m)& x\in [a_m,b]\\ 0 & x \notin [a_m,b] \end{cases}
\]
Then \ref{LC2.i} holds since $\int_{-\infty}^\infty f(x)\,dx=\varepsilon$
and to satisfy \ref{LC2.ab} we may take the intervals 
$[a_i,b_i]:=[a_i,b]$, $i=0,\dots,m-1$.
Indeed, with this choice $\sum_i\int_{a_i}^{b_i} f(x)\,dx =m\varepsilon>1/\varepsilon$ and
$a_i+\tau(b_i-a_i)< a_i+\sigma (b-a_i)=a_{i+1}$, and so the intervals
$[a_i,a_i+\tau(b_i-a_i)]$ are mutually disjoint.
\end{proof}

\begin{thm}\label{M.1}
There is a Lebesgue integrable function $f:\R\to[0,\infty)$ 
that is, for any $0<\alpha<1$, not \MC{\alpha} integrable on any interval $I\subset\R$.
\end{thm}

\begin{proof}
Let $r_k$ be an enumeration of all rational numbers, 
$\varepsilon_k:= 2^{-k}$ and $\tau_k:=1-2^{-k}.$ 
Obtain $f_k$ from Lemma~\ref{LC2} used with these
$\varepsilon_k,\tau_k$ and  $I_k=(r_k-\varepsilon_k,r_k+\varepsilon_k)$. 
Since the function $\sum_{k=1}^\infty f_k$ is Lebesgue integrable, it is finite almost everywhere.
Hence the set $N:=\{x\in\R: \sum_{k=1}^\infty f_k(x)=\infty\}\cup\{r_k:k=1,2,\dots\}$
is Lebesgue null.
Choose a $G_\delta$ Lebesgue null set $G$ containing $N$ and define
\[f (x):= \begin{cases} \sum_{k=1}^\infty f_k(x) & x\notin G,\\ 0 & x\in G.\end{cases}\]
Then $f:\R\to[0,\infty)$ is Lebesgue integrable; let $F$ be its indefinite Lebesgue integral.

Suppose $0<\alpha<1$ and $f$ is \MC{\alpha} integrable on some interval $I$.
By Proposition~\ref{T}, $f$
is an \MC{\alpha} derivative of $F$. 
Let
$\varphi$ be an $\alpha$-control function for $(F,f)$. 
For each $x\in N$ find $\delta_x>0$ such that
\[|F(x+h)-F(x)|\le \varphi(x+\alpha h)-\varphi(x)\]
whenever $0<h<\delta_x$.

Let $N_l:=\{x\in G: \delta_x>2^{-l}\}$
and infer from the Baire Category Theorem that
there are $l\in\N$ and an interval $J=(u,v)\subset I$  
such that $N_{l}\cap J$ is dense in $J$.
Since $\varepsilon_k\to 0$ and $\tau_k\to 1$, the inequalities
$\alpha<\tau_k$ and $\varphi(u)-\varphi(v)<1/\varepsilon_k$
hold for all sufficiently large~$k$. Using also that 
$J$ contains $[r_k-\varepsilon_k,r_k+\varepsilon_k]$
for infinitely many $k$, we find
$k>l$ such that $\alpha<\tau_k$, $\varphi(u)-\varphi(v)<1/\varepsilon_k=2^k$
and $[a,b]:=[r_k-\varepsilon_k,r_k+\varepsilon_k]\subset J$.

Recalling that $f_k$ has been found by using Lemma~\ref{LC2},
we infer from its statement~\ref{LC2.ab}
that there is a finite collection of  
intervals $[a_i,b_i]\subset (a,b)$ such that the intervals $[a_i,a_i+\tau_k(b_i-a_i)]$
are mutually disjoint and $\sum\int_{a_i}^{b_i} f_k(x)\,dx>1/\varepsilon_k=2^k$.
Since $N_l\cap(a,b)$ is dense in $(a,b)$ and $\alpha<\tau_k$, there are
$x_i\in(a,b)$, $x_i\le a_i$ close enough to $a_i$ such that the intervals 
$[x_i,x_i+\alpha(b_i-x_i)]$ are also mutually disjoint.
Recalling that $k>l$, we see that
$b_i-x_i\le 2\varepsilon_k\le 2^{-l}$, so $\delta_{x_i}>2^{-l}\ge b_i-x_i$
and the definition of $\delta_{x_i}$ implies 
\[\varphi(x_{i} +\alpha(b_i-a_i))-\varphi(x_{j})\ge F(b_{i}) - F(x_{i}).\]
Since $\varphi$ is increasing and 
$[x_i,x_i+\alpha(b_i-x_i)]$ are mutually disjoint 
subintervals of $(u,v)$, we get
\begin{align*}
2^k>\varphi(v)-\varphi(u)&\ge \sum_i(\varphi(x_{i} +\alpha(b_i-a_i))-\varphi(x_{j})) \\
&\ge\sum_i (F(b_{i}) - F(x_{i})) \ge \sum_i\int_{x_i}^{b_i} f_k(x)\,dx
>\sum_i\int_{a_i}^{b_i} f_k(x)\,dx>2^k,
\end{align*}
which is the desired contradiction.
\end{proof}

The final result of this section may seem to be a good start for showing the equivalence 
of Denjoy-Perron and \MC{\alpha} integrals for any $\alpha\ge 1$ by the methods
stemming from Denjoy's original definition \cite{AD} of his integral.  
These methods have been used, for example,
to show equivalence of the Denjoy and Perron integrals
(see the proof of the Theorem of Hake-Alexandroff-Looman at \cite[VIII(\S3)]{Saks}). 
However, the use of this idea to show equivalence of the \MC{\alpha} and Denjoy-Perron integrals 
would need a similar statement with $I$ replaced by
any closed subset of $I$. This turned out to be possible when $1\le\alpha\le2$,
and will lead us to the proof of Theorem~\ref{M.2}, but
for $\alpha>2$ we actually used this information to guess how the 
counterexamples required in Theorem~\ref{MC}\ref{MC.3} and~\ref{MC.4} may look like.

\begin{prp}\label{dense}
If, for some $\alpha>0$, 
$F$ is an indefinite \MC{\alpha} integral of $f$ on an interval $I$, there is a dense 
open subset $G$ of $I$ such that $F$ is absolutely continuous on every component
of $G$.  Consequently,  $F$ is an indefinite Lebesgue integral of $f$ on every 
component of~$G$.
\end{prp}

\begin{proof}
It suffices to show that there is an interval $J\subset I$ on which
$F$ has bounded variation. Indeed, since 
$f=F'$ a.e., we infer from \cite[Lemma IV(7.4)]{Saks}
that $f$ is Lebesgue integrable on $(a,b)$ and from 
Theorem~\ref{BMT} that $F$ is an indefinite Lebesgue integral of $f$ on $J$.
The required set $G$ is then obtained by a repeated use of this for 
suitable subintervals of $I$.

Let $\varphi$ be an $\alpha$-control function
for $(F,f)$ and
for every $x\in I$ choose $\delta_x>0$ such that
$|F(y)-F(x) - f(x)(y-x)| \le |\varphi(x+\alpha(y-x))-\varphi(x)|$ when $|y-x|\le\delta_x$.
Let $Q_k:=\{x\in I: |f(x)|\le k,\,\delta_x>1/k\}$.
Then $Q=\bigcup_{k=1}^\infty Q_k$, and hence by the Baire Category Theorem
there is an open interval $(a,b)\subset I$ such that
$Q_k\cap (a,b)$ is dense in $(a,b)$.  
We diminish~$(a,b)$, if necessary, so that $b-a<1/k$ and show
\begin{equation}\label{PE.1}
|F(y)-F(x)| \le |\gamma(y)-\gamma(x)| \text{ for $x,y\in(a,b)$.}
\end{equation}
where $\gamma(x) = 2n(\varphi(x)+k x)$ and $n\in\N$ is 
the least integer greater than $3\alpha$.

To prove \eqref{PE.1},
suppose $x,y\in(a,b)$, $x<y$ and $\varepsilon>0$,
and find $0<\delta<(y-x)/6(1+\alpha)$ such that $|F(u)-F(v)|<\varepsilon/2$
when $v=x,y$ and $|u-v|<\delta$.
Use that $Q_k$ is dense in $(x,y)$ to find
$x< x_0<x_1<\dots<x_{2n}< y$ such that
$x_i\in Q_k$ and $|x_i-(x+i(y-x)/2n)|<\delta$.
If $1\le i\le n$, we have 
$x_{i-1}+\alpha(x_i-x_{i-1})<x+(i-1)(y-x)/2n +\delta +\alpha((y-x)/2n +\delta) <
x+(y-x)/2 +(y-x)/6 + (y-x)/6 + (y-x)/6=y$.
Since $x_{i-1}\in Q_k$  and $x_i-x_{i-1}<1/k<\delta_{x_i}$, this and monotonicity of $\varphi$ imply
\begin{multline*}
|F(x_i)-F(x_{i-1})|\le \varphi(x_{i-1}+\alpha(x_i-x_{i-1}))-\varphi(x_{i-1}) + k(x_i-x_{i-1})\\
\le \varphi(y)-\varphi(x) +k (y-x)\le (\gamma(y)-\gamma(x))/2n.
\end{multline*}
If $n< i\le 2n$, we similarly have
$x_i-\alpha(x_i-x_{i-1}) \ge x$, and hence
\[|F(x_i)-F(x_{i-1})|\le |\varphi(x_i-\alpha(x_i-x_{i-1}))-\varphi(x_{i})| + k(x_i-x_{i-1}) 
\le (\gamma(y)-\gamma(x))/2n.\]
Summing these inequalities gives
$|F(x_{2n})-F(x_0)|\le \gamma(y)-\gamma(x)$, from which we infer
\[|F(y)-F(x)|\le |F(x_{2n})-F(x_0)| +|F(x_{2n})-F(y)| + |F(x)-F(x_0)| \le 
\gamma(y)-\gamma(x)+2\varepsilon\]
and, since $\varepsilon>0$ is arbitrary,
$|F(y)-F(x)|\le |\gamma(y)-\gamma(x)|$. 

Having thus finished the proof of \eqref{PE.1}, we use
it together with the monotonicity of $\varphi$ to infer that 
\[\sum_i|F(u_i)-F(v_i)|\le \sum_i (\gamma(v_i)-\gamma(u_))\le 
\gamma(b)-\gamma(a)\]
for any mutually disjoint intervals $(u_i,v_i)$ with end-points in $(a,b)$.
Consequently, $F$~has bounded variation on $(a,b)$, which, as
explained at the beginning of this proof, implies the statement of the Proposition.
\end{proof}

\subsection{Coincidence of \MC{} and Denjoy-Perron integrals}

The coincidence the \MC{} and Denjoy-Perron integrals,
and so of the \MC{\alpha} and Denjoy-Perron integrals when $\alpha=1$ was proved
in \cite[Theorem~3]{BM} by Bendov\'a and Mal\'y. In their proof they used modern
definitions of the Denjoy-Perron integral. Here we point out that it
can be proved directly by using Perron's original definition of the Denjoy-Perron integral.
For this, we first
briefly recapitulate the definition of the indefinite Perron integral.
Full information may be found in \cite[VI(\S6) and VIII(\S3)]{Saks}.

\begin{dfn}\label{perron}
Let $f$, $F$ be functions defined on an open interval $I$.
Then $F$ is said to be an indefinite Perron integral of $F$ on $I$
if for every $\varepsilon>0$ and a compact interval $J\subset I$ there are function $U,V:I\to\R$ 
such that
$\underline D U(x)\ge f(x)\ge \overline D V(x)$ for every $x\in I$
and $|U(x)-F(x)|+|V(x)-F(x)|<\varepsilon$ for every $x\in J$.
Here
\[\underline D U(x) := \liminf_{y\to x} \frac{F(y)-F(x)}{y-x}
\text{ and }
\overline D U(x) := \limsup_{y\to x} \frac{F(y)-F(x)}{y-x}.\]
\end{dfn}

\begin{thm}[Bendov\'a and Mal\'y]\label{BMT}
$F$ is an indefinite \MC{} integral of $f$ on $I$ if and only it is its
indefinite Perron integral on $I$.
\end{thm}

\begin{proof}
Write $I$ as a union of an increasing sequence of compact intervals $J_k=[a_k,b_k]$.

Let $F$ be an indefinite \MC{} integral of $f$ on $I$ and 
$\varphi$ be a control function for $(F,f)$. Given $\varepsilon>0$,
we let $U_\varepsilon:=F+\varepsilon\varphi$ and $V_\varepsilon:=F-\varepsilon\varphi$.
For every $x\in I$ there is $\delta>0$ such that
$|y-x|<\delta$ implies
$|F(y)-F(x)-f(x)(y-x)|\le \varepsilon |\varphi(y)-\varphi(x)|$.
Rearranging this inequality gives
$(U_\varepsilon(y)-U_\varepsilon(x))/(y-x) \ge f(x)$;
hence $\underline{D} U_\varepsilon(x)\ge f(x)$.
A symmetric argument shows that $\overline{D} V_\varepsilon(x) \le f(x)$.
Finally, on each $J_k$ we have 
$|U(x)-F(x)|+|V(x)-F(x)|\le 2\varepsilon\max(|\varphi(a_k)|,|\varphi(b_k)|)$.
Hence $F$ is an indefinite Perron integral of $f$ on~$I$.

Assuming $F$ is an indefinite Perron integral of $f$ on $I$, for every $k\in\N$
there are functions $U_k, V_k$ 
such that $\underline D U_k\ge f\ge \overline D V_k$ on $I$ and 
$|U_k-V_k|\le 2^{-k}$ on $J_k$. Since $\underline D (U_k-V_j)\ge 0$,
$U_k-V_j$ are increasing (see, for example, \cite[Theorem VI(3.2)]{Saks}),
and so are also $U_k-F$ and $F-V_j$. Letting 
$\varphi(x):= x+\sum_{k=1}^\infty k (U_k(x)-V_k(x))$, which is well-defined
since each $x\in I$ belongs to all but finitely many $J_k$,
we show that $\varphi$ is the required control function. 
Clearly, it is strictly increasing.
For every $x\in I$ and $k\in\N$ there is
$\delta>0$ such that  $0<|y-x|<\delta$ implies
\[\frac{U_k(y)-U_k(x)}{y-x}\ge f(x)-1/k\ \text{ and }\
\frac{V_k(y)-V_k(x)}{y-x}\le f(x)+1/k.\]
Hence
\begin{align*}
\frac{F(y)-F(x)}{y-x}-f(x)&\le \frac{U_k(y)-U_k(x)}{y-x}-\frac{V_k(y)-V_k(x)}{y-x}+\frac1k\le 
\frac{\varphi(y)-\varphi(x)}{k(y-x)}\\
\intertext{and a symmetric argument gives}
\frac{F(y)-F(x)}{y-x}-f(x)&\ge \frac{V_k(y)-V_k(x)}{y-x}-\frac{U_k(y)-U_k(x)}{y-x}-\frac1k\ge 
-\frac{\varphi(y)-\varphi(x)}{k(y-x)},
\end{align*}
which shows that $\varphi$ is a control function for $(F,f)$.
\end{proof}

\subsection{Relation between \MC{\alpha} and Denjoy-Perron integrals}

Our main argument here shows that every \MC{2}-integrable
function is Denjoy-Perron integrable, which together with the
result of  Bendov\'a and Mal\'y and Proposition~\ref{T} immediately implies that  
for $1\le\alpha\le 2$ the \MC{\alpha} and Denjoy-Perron integrals coincide.

\begin{thm}\label{M.2}
For any $1\le\alpha\le2$, the \MC{\alpha} integral on any interval $I$ 
coincides with the Denjoy-Perron integral.
\end{thm}

\begin{proof}
By Theorem~\ref{BMT} and Proposition~\ref{T} it suffices to show that
every \MC{2} integrable function is Denjoy-Perron integrable.
So suppose $F$ is an indefinite \MC{2} integral of $f$
on~$I$. Denote by $G$ the union of those open subintervals of $I$ on 
which $F$ is \ACGs. By the Lindel\"of property of the real line,
$G$ is the union of a countable family of such subintervals, 
and so $F$ is \ACGs\ on $G$. Since $F'=f$ almost everywhere,
$F$ is an indefinite Denjoy-Perron integral of $f$
on each component of $G$.
This means that if $G=(a,b)$, we are done; so assume
$Q:=(a,b)\setminus G\ne\emptyset$. 
Notice that $Q$ has no isolated points by \cite[Lemma VIII(3.1)]{Saks}.
We show that there is
an interval $[a,b]\subset I$ such that $Q\cap(a,b)\ne\emptyset$
and
\begin{enumerate}[label=(\alph*)]
\item\label{P2.a}
$f$ is Lebesgue integrable on $Q$;
\item\label{P2.b}
the series of the oscillations of $F$ on the components
of $(a,b)\setminus Q$ converges.
\end{enumerate}
By Lemma~3.4 in \cite[Chapter 7]{Saks} this will imply that
$F$ is an indefinite Denjoy-Perron integral of $f$ on $(a,b)$. But then
$(a,b)\subset G$, contradicting $Q\cap(a,b)\ne\emptyset$.

To find the interval $(a,b)$, let $\varphi$ be a 2-control function
for $(F,f)$ and
for every $x\in I$ choose $\delta_x>0$ such that
$|F(y)-F(x) - f(x)(y-x)| \le |\varphi(x+2(y-x))-\varphi(x)|$ when $|y-x|\le\delta_x$.
Let $Q_k:=\{x\in Q: |f(x)|\le k,\,\delta_x>1/k\}$.
Then $Q=\bigcup_{k=1}^\infty Q_k$, and hence by the Baire Category Theorem
there is an open interval $(a,b)$ such that $M:=Q\cap (a,b)\ne\emptyset$ and 
$Q_k\cap M$ is dense in $M$.  
Since $Q$ has no isolated points, we may diminish~$(a,b)$, if necessary,
to guarantee $a,b\in M$ and $|b-a|<1/k$.
We show that
\begin{equation}\label{PE.2}
|F(y)-F(x)| \le 2|\varphi(y)-\varphi(x)|+2k|y-x| \text{ for $x,y\in M$.}
\end{equation}
To prove this we use
 that $|y-x|<\delta_x$ to infer that
 \[|F((x+y)/2)-F(x)| \le |\varphi(y)-\varphi(x)|+|f(x)||y-x|\le |\varphi(y)-\varphi(x)|+k|y-x|.\]
 Similarly we have 
 $|F((x+y)/2)-F(y)| \le |\varphi(y)-\varphi(x)|+k|y-x|$.
 Hence 
\[|F(y)-F(x)|\le |F((x+y)/2)-F(x)|+|F((x+y)/2)-F(y)| \le 
 2|\varphi(y)-\varphi(x)|+2k|y-x|,\] 
as claimed.
 
Clearly, monotonicity of $\varphi$ and  \eqref{PE.2} show that 
\[\sum_i|F(u_i)-F(v_i)|\le \sum_i (2|\varphi(v_i)-\varphi(u_i)|+2k|v_i-u_i|)\le 
2|\varphi(b)-\varphi(a)|+2k|b-a|\]
for any mutually disjoint intervals $(u_i,v_i)$ with end-points in $M$.
Consequently, $F$ has bounded variation on $M$ and, since it is
continuous, also on $Q$. (See \cite[VII(\S4)]{Saks}). Since 
$f=F'$ a.e., we infer~\ref{P2.a} from \cite[Lemma VIII(2.1)] {Saks}
and \cite[Lemma IV(7.4)]{Saks}.

For \ref{P2.b}, consider any component $(u,v)$ of $(a,b)\setminus Q$.
If $x\in[u,(u+v)/2]$, then 
\[|F(x)-F(u)|\le|\varphi(u+2(x-u))-\varphi(u)| \le \varphi(v)-\varphi(u),\]
and if $x\in[(u+v)/2,v]$, then 
\[|F(x)-F(v)|\le|\varphi(v+2(x-v))-\varphi(v)| \le \varphi(v)-\varphi(u).\]
Using also \eqref{PE.1}, we get
\[|F(x)-F(u)|\le 3(\varphi(v)-\varphi(u))+2k(y-x),\] and hence the
oscillation of $F$ on $(u,v)$ is at most
$6(\varphi(v)-\varphi(u))+4k(y-x)$. The sum of these oscillations
over the components of $(a,b)\setminus Q$ is therefore at most
\[6(\varphi(b)-\varphi(a))+4k(b-a),\] 
showing that \ref{P2.b} holds and so proving the Theorem.
\end{proof}

\subsection{Preliminaries to constructions for $\alpha>2$}\label{prlm}

As we already said in connection with Proposition~\ref{dense}, 
the examples we have to construct to show
the cases of Theorem~\ref{MC} when $\alpha>1$ need a closed, necessarily nowhere dense,
set on which the function we construct is not \ACGs, respectively \ACG. The (ternary) Cantor set is a good candidate,
especially when we recall that the corresponding Cantor function
has a very quick increase when passing through the Cantor set, thereby
providing a good choice for a control function. We will actually use the ternary Cantor set
to prove Theorem~\ref{MC}\ref{MC.3}, but for the proof of Theorem~\ref{MC}\ref{MC.4}
we need a bit more room that we gain by using one of the Cantor type sets with base~5. We therefore 
fix notation for the construction of these sets with any odd base $q\ge 3$, although we
will use it only for $q=3$ and $q=5$.

The Cantor type sets we use are defined in a standard way.
Let $q=2m+1$ be an odd integer. We recursively
define collections of closed intervals $\I_k$, $k=0,1,2,\dots$ and
collections of open intervals $\J_k$, $k=1,2,\dots$ as follows.

We let $\I_0:=\{[0,1]\}$ and, whenever $\I_{k-1}$ has been defined,
and $[u,v]\in\I_{k-1}$, we put into $\I_k$ the intervals
$[u+j (v-u)/q,u+(j+1)(v-u)/q]$ where $0\le j\le q-1$ is even, and into $\J_k$
the intervals $(u+j (v-u)/q,u+(j+1)(v-u)/q)$ where $0\le j\le q-1$ is odd.

The (base $q$) Cantor set is then defined as $C:=\bigcap_{k=0}^\infty C_k$
where $C_k$ is the union of intervals from $\I$. We will use the following straightforward facts.
\begin{itemize}
\item $C_0\supset C_1\supset\dots$
\item$\I_k$ is the set of (connected) components of~$C_k$.
\item $\J_k$ is the set of components of $C_k\setminus C_{k-1}$.
\item The set of components of $[0,1]\setminus C$ is
$\J:=\bigcup_{k=1}^\infty\J_k$, where the union is disjoint.
\item
For every $(u,v)\in\J_k$, $u,v\in C$
and both $u,v$ are end-points of intervals from $\I_k$.
\item $\I_k$ consists of $(m+1)^k$ mutually disjoint closed intervals of length $q^{-k}$.
\item $\J_k$ consists of $m(m+1)^{k-1}$ mutually disjoint open intervals of length $q^{-k}$.
\item If $k>p$, an interval from $\I_p$ contains $m(m+1)^{k-p-1}$ intervals from $\J_k$.
\end{itemize}

We will also use the corresponding Cantor function $\psi:\R\to[0,1]$, which is 
characterized by being continuous,
increasing, constant on each component of $\R\setminus C$, 
and, for each $k$, mapping each interval from $\I_k$ onto an interval of length
$(m+1)^{-k}$. 
To define it, we may, for example, let $\psi_k(x):=(m+1)^{-k}q^k |(-\infty,x)\cap C_k|$,
observe that $|\psi_k-\psi_{k-1}|\le(m+1)^{-k}$ and define $\psi:=\lim_{k\to\infty}\psi_k$.

\subsection{\MC{\alpha} integrabilites differ for $\alpha\ge 2$}

In this proof we will use the ternary Cantor set $C$, the collections of intervals
$\J_k$ and~$\J$ and the Cantor function~$\psi$,
as described in Section~\ref{prlm} when we take $q=3$.

\begin{lem}\label{C}
There are $Q_J>0$, $J\in\J$, such that
\begin{enumerate}
\item\label{C.c}
$\lim_{k\to\infty} \max\{Q_j: J\in\J_k\} =0$;
\item\label{C.a}
$\sum_{J\in\J,\,J\subset I}  Q_J=\infty$
whenever $I\subset(0,1)$ is an open interval meeting $C$.
\item\label{C.b}
for every $\eta>0$, 
\begin{equation}\label{EC.b}
Q_J \le \eta\min(\psi(b+\eta(b-a))-\psi(b),\psi(a)-\psi(a-\eta(b-a)))
\end{equation}
holds for all but finitely many $J=(a,b)\in\J$;
\end{enumerate} 
\end{lem}

\begin{proof}
We let $k_l := 4^l$,
define $Q_J:=2^{-k-2l}$ when $J\in\J_k$ and $k_{l-1}\le k<k_{l}$
and notice that \ref{C.c} holds.

For any open interval $J$ meeting $C$ there is $p\in\N$ such that
$J$ contains an interval from $\I_p$.
It follows that for $k>p$, $J$ contains at least $2^{k-p-1}$ intervals from~$\J_k$.
Choose $m\in\N$ such that $4^{m-1}>p$. Then,
if $l\ge m$ and $k_{l-1}\le k<k_l$, 
the sum of $Q_J$ over those $J\in\J_k$ that are contained in $J$
is at least $2^{-p-2l-1}$. Hence
\[\sum_{J\in\J,\,J\subset J}  Q_J
\ge \sum_{l=m}^\infty\sum_{k=k_{l-1}}^{k_l-1} 2^{-p-2l-1}
= \sum_{l=m}^\infty 2^{-p-2l-1}(4^l-4^{l-1}) \ge \sum_{l=m}^\infty 2^{-p-3}=
\infty.\]

To prove \ref{C.b}, suppose $\eta>0$ and choose $n\in\N$ such that $2^{-n}<\eta$.
Whenever $l\ge n$, $k_{l-1}\le k<k_l$ and $J=(a,b)\in\J_k$,
we notice that  $b+\eta(b-a)\ge b+2^{-n} 3^{-k} \ge b+3^{-k-n}$, and so that
$b$ is the left-end point of an interval from $\I_{k+n}$ of length $3^{-k-n}$
that is mapped by $\psi$ onto an interval of length $2^{-k-n}$. Hence
\[\eta(\psi(b+\eta(b-a))-\psi(b)) \ge 2^{-k-2n}\ge Q_J.\] 
A symmetric argument shows $\eta(\psi(a)-\psi(a-\eta(b-a))) \ge Q_J$. 
Hence \eqref{EC.b} holds for
all $J\in\J_k$ when $k\ge n$,
hence for all $J\in\J$ except finitely many.
\end{proof}

\begin{thm}\label{M.3}
For any interval $I$ and $\alpha\ge 2$ there are functions $f,F:\R\to\R$ such that
\begin{enumerate}
\item 
$F$ is the Denjoy-Khintchine indefinite integral of $f$ on $\R$;
\item
$F$ is the \MC{\beta} indefinite integral of $f$ on $\R$ for every $\beta>\alpha$;
\item
$f$ is not \MC{\alpha} integrable on $I$.
\end{enumerate}
\end{thm}

\begin{proof}
Without loss of generality we assume that $I$ is an open interval containing $[0,1]$. 
Let $Q_J$ be as in Lemma~\ref{C}.
Choose a continuously differentiable function $\xi:\R\to[0,1]$ 
with support in $[-1,1]$ such that $\xi(x)=1$ for $-1/2\le x\le1/2$.
Let $\sigma:=1/\alpha$ and $\sigma_J:= 3^{-k}\sigma$ for $J\in\J_k$.
For $J=(a,b)\in\J$ let $u_J:= a+\sigma (b-a)$ and define
\[
\xi_J(x):= Q_J \xi((x-u_J)/(\sigma_J(b-a)))\text{ and  }
F(x):=\sum_{J\in\J} \xi_J(x).
\]
Notice that the support of $\xi_J$ is contained in 
$[u_J-\sigma_J (b-a),u_J+\sigma_J (b-a)]$ which, since $\sigma_J<\sigma\le1/2$
is contained in $J$. 
Since the intervals from~$\J$ are mutually disjoint,
the functions
$F_k(x):=\sum_{l=1}^k\sum_{J\in\J_l} \xi_J(x)$ 
satisfy $|F-F_k|\le \max_{l>k}\max\{Q_J: J\in J_k\}$. Hence
Lemma~\ref{C}\ref{C.c} implies that 
$F_k$ converge to $F$ uniformly,
and since $F_k$ are continuous, so is $F$. Moreover,
$F=0$ outside $[0,1]$ and $F=\xi_J$ on $J\in\J$; hence
$F$ is continuously differentiable on $\R\setminus C$ and we may define
\[ f(x):= \begin{cases} F'(x) & \text{ when $x\in\R\setminus C$}\\
                                    0 & \text{ when $x\in C$.}
               \end{cases}
\]

The first statement, that $F$ is the Denjoy-Khintchine indefinite integral of $f$
is straightforward. Since $F=0$ on $C$, it is absolutely continuous on $C$ and since it is
continuously differentiable outside $C$, it is \ACG\ also on $\R\setminus C$.
So $F$ is \ACG\ on $\R$, continuous and $F'=f$
almost everywhere, which implies that indeed 
$F$ is the Denjoy-Khintchine indefinite integral of $f$.

To prove the second statement, we let $\varphi(x):=x+\psi(x)$,
observe that $\varphi$ is strictly increasing and hence it suffices to show that
for every $\beta>\alpha$, $x\in\R$ and $\varepsilon>0$ there is $\delta>0$ such that 
\begin{equation}\label{Eb}
|F(y)-F(x)-f(x)(y-x)|\le \varepsilon |\varphi(x+\beta(y-x))-\varphi(x)|
\end{equation}
whenever $y\in(x-\delta,x+\delta)$. So we
fix $\beta>\alpha$, $x\in\R$ and $\varepsilon>0$, and 
find such a $\delta$.
This is easy when $x\notin C$, since then $f(x)=F'(x)$ and so for small enough $\delta>0$,
\[|F(y)-F(x)-f(x)(y-x)|\le\varepsilon |y-x|\le \varepsilon|\varphi(x+\beta(y-x))-\varphi(x)|\]
for every $y\in(x-\delta,x+\delta)$. 

So assume $x\in C$.
Choose $\eta>0$ such that $0<\eta<\varepsilon$ and,
recalling that $\sigma=1/\alpha$ and $\beta>\alpha$, that also
$\beta\sigma>(1+\eta)$.
By Lemma~\ref{C}\ref{C.b} we find
$0<\kappa<\sigma-(1+\eta)/\beta$  such that for every $J=(a,b)\in\J$ with $b-a<\kappa$,
\begin{equation}\label{EQ1}
Q_J \le \eta\min\bigl(\psi(b+\eta(b-a))-\psi(b),\psi(a)-\psi(a-\eta(b-a))\bigr).
\end{equation}
Since $\sigma_J\le b-a \le\kappa$, we also have
\begin{equation}\label{EQ2}
\beta(\sigma-\sigma_J)\ge 1+\eta.
\end{equation}

Let $\delta:=\kappa\sigma/2$ and
consider any $y\in(x-\delta,x+\delta)$. 
Since $F\ge 0$ and $f(x)=F(x)=0$, \eqref{Eb} reduces to showing that
\begin{equation}\label{Ed}
F(y)\le \varepsilon |\varphi(x+\beta(y-x))-\varphi(x)|.
\end{equation}
As this is obvious when $F(y)=0$, we assume that $F(y)\ne 0$.
Then there is $J=(a,b)\in\J$
such that $0< F(y)\le Q_J$
and $y\in[a+(\sigma-\sigma_J)(b-a),a+(\sigma+\sigma_J)(b-a)]$.
In particular we have $|x-y|\ge\sigma(b-a)/2$ since $x\notin J$, $\sigma-\sigma_J\ge\sigma/2$ 
and $1-(\sigma+\sigma_J)\ge 1/2-\sigma_J\ge\sigma-\sigma_J\ge\sigma/2$.
Hence $\kappa\sigma/2=\delta > |y-x| \ge \sigma (b-a)/2$, which
gives $b-a<\kappa$ and hence \eqref{EQ1} 
and \eqref{EQ2} hold.

When $\alpha>2$ the situations when $y>x$ and $y<x$ are not completely symmetric,
so we continue by distinguishing these two cases.

\begin{cs}{$y>x$}
Then $x\le a<y$ and~\eqref{EQ2} gives
\begin{multline*}
x+\beta(y-x) 
= x+ \beta(a-x)+ \beta(y-a)\\
\ge x + (a-x)+\beta(y-a)
\ge a +\beta(\sigma-\sigma_J) (b-a) \ge b+\eta(b-a).
\end{multline*}
Hence 
\begin{multline*}
F(y)\le Q_J \le \eta (\psi(b+\eta(b-a))-\psi(b))\\
\le \eta( \psi(x+\beta(y-x)) -\psi(x))
\le \varepsilon( \varphi(x+\beta(y-x)) -\varphi(x)).
\end{multline*}
\end{cs}

\begin{cs}{$y<x$} 
Then $x\ge b$ and, using $\sigma\le1/2$ and~\eqref{EQ2}
to infer \[\beta(1-\sigma-\sigma_J) \ge \beta(\sigma-\sigma_J)\ge\eta,\] 
we get
\begin{multline*}
x+\beta(y-x) 
= x - \beta(x-b)- \beta(b-y)\\
\le x - (x-b)-\beta(b-y)
\le b - \beta(1-\sigma-\sigma_J) (b-a) 
\le a-\eta(b-a).
\end{multline*}
Hence 
\begin{multline*}
F(y)\le Q_J \le \eta (\psi(a) - \psi(a-\eta(b-a)))\\
\le \eta( \psi(x) - \psi(x+\beta(y-x)))
\le \varepsilon|\varphi(x+\beta(y-x)) -\varphi(x)|. 
\end{multline*}
\end{cs}

In both cases we have proved that \eqref{Ed} holds for $y\in(x-\delta,x+\delta)$ as required, and we
conclude that $f$ is \MC{\beta} integrable. 

It remains to show that $f$ is not \MC{\alpha} integrable on $(0,1)$. Arguing by
contradiction and using Proposition~\ref{T}, we 
assume that $F$ is an indefinite \MC{\alpha} integral of $f$ on $(0,1)$.
Then there is a strictly increasing function $\gamma:(0,1)\to\R$ such that for every $x\in(0,1)$
there is $\delta_x>0$ such that for every $y\in (x,x+\delta_x)$,
\[|F(y)-F(x)-f(x)(y-x)|\le \gamma(x+\alpha(y-x))-\gamma(x).\]

Let $\Delta_k:=\{x\in C\cap(0,1): \delta_x>1/k\}$.
Since $C=\bigcup_{k=1}^\infty \Delta_k$, the Baire Category Theorem
implies that there are $k\in\N$ and an open interval $J\subset(0,1)$ 
such that $J\cap C\ne\emptyset$ and
$\Delta_k\cap J$ is dense in $C\cap J$. 
We diminish $J$ if necessary to achieve $|J|<1/k$ and
choose an interval $[a,b]\subset J$ such that $(a,b)\cap C\ne\emptyset$.

By Lemma~\ref{C}\ref{C.a} find $n\in\N$ and
intervals $J_i=(a_i,b_i)\in\J$, $i=1,\dots n$ such that $(a_i,b_i)\subset (a,b)$ and
$\sum_{i=1}^n Q_{J_i} > \gamma(b)-\gamma(a)$. 
Using that $a_i+\sigma(b_i-a_i)=u_{J_i}$, $C$ has no isolated points and $a_i\in C\cap J$, 
we find  $x_i\in C\cap J\cap\Delta_k$ so 
close to $a_i$ that the intervals $(x_i,b_i)$, 
$i=1,\dots,n$ are mutually disjoint and 
\[y_i:=x_i+\sigma(b_i-x_i)\in [u_{J_i}-\sigma_{J_i}(b_i-a_i)/2,u_{J_i}+\sigma_{J_i}(b_i-a_i)/2].\]
Hence $F(y_i)=Q_{J_i}$ and since $F(x_i)=f(x_i)=0$, 
\[Q_{J_i} = F(y_i)-F(x_i)-f(x_i)(y_i-x_i)
\le \gamma(x_i+\alpha(y_i-x_i))-\gamma(x_i) = \gamma(b_i)-\gamma(x_i),\]
Finally, we use that $\gamma$ is increasing to get 
\[\gamma(b)-\gamma(a)\ge \sum_{i=1}^n (\gamma(b_i)-\gamma(x_i)) \ge \sum_{i=1}^n Q_{J_i}
>\gamma(b)-\gamma(a),\] 
which is the desired contradiction.
\end{proof}

\subsection{Relation between \MC{\alpha} and Denjoy-Khintchine integrals}

Since indefinite \MC{\alpha} integrals are differentiable almost everywhere by 
Proposition~\ref{T}\ref{T.2} but indefinite Denjoy-Khintchine integrals need 
not be differentiable almost everywhere (as an example one may take a 
continuous function that is everywhere 
approximately differentiable but not differentiable almost everywhere), we have

\begin{obs}
On any interval $I$ there is a Denjoy-Khintchine integrable 
functions that is not \MC{\alpha} integrable 
for any $\alpha>0$.
\end{obs}

For the opposite direction we need more work. 
To construct the required example,  we will use the Cantor-type set $C$ with base $q=5$, 
its approximating sets $C_k$, the sets of intervals $\I_k$, $\J_k$ and $\J$, 
and the corresponding Cantor-type function $\psi$
described in Section~\ref{prlm}.

\begin{thm}\label{M.4}
There is a function $f:\R\to\R$ that is \MC{\alpha} integrable on $\R$ for every $\alpha>2$
but is not Denjoy–Khintchine integrable on $[0,1]$. 
\end{thm}

\begin{proof}
For any interval $I=[a,b]$ and $0<\tau< 1/2$ choose a continuously differentiable function
$g_{I,\tau}:\R\to\R$ which is increasing on $(-\infty,(a+b)/2]$, 
decreasing on $[(a+b)/2,\infty)$ and satisfies
\[g_{I,\tau}(x) =
\begin{cases}
0 & x\le a+(1+\tau)(b-a)/5;\\
0 & x \ge b-(1+\tau)(b-a)/5;\\
1 & x\in [a+(2-\tau)(b-a)/5,b-(2-\tau)(b-a)/5].
\end{cases}\]

We let $\sigma_k:= 1/(k+1)$, $\tau_k:= (k+1)/(2(k+2))$ and
\[F(x):=\sum_{k=1}^\infty \sigma_k 3^{-k}\sum_{I\in \I_{k-1}}  g_{I,\tau_k}(x).\]
Since for each $k$, $\sum_{I\in \I_k}  g_{I,\tau_k}(x)$ is continuous and bounded by one,
$F$ is the sum of a uniformly convergent series of continuous functions,
and so it is continuous. We list the following easy but crucial 
properties of $F$.

\begin{enumerate}[label=(\alph*)]
\item\label{F.2}
$|F(y)-F(x)| \le \sigma_k 3^{-k+1}$ whenever $x,y$ belong to the same interval from $\I_{k-1}$;
\item\label{F.4}
$|F(b)-F(a)|=\sigma_k 3^{-k}$ whenever $(a,b)$ is an interval from $\J_k$;
\item\label{F.5}
if $(a,b)\in\J_k$, then $F$ is constant
on $[a,a+\tau_k(b-a)]$ as well as on $[b-\tau_k(b-a),b]$.
\item\label{F.3}
On each $J\in\J_{k}$ the sum defining $F$ is finite,
and hence $F$ is continuously differentiable on $\R\setminus C$.
\end{enumerate}

The property \ref{F.3} allows us to define
\[f(x):=
\begin{cases} F'(x) &x\notin C;\\
0&x\in C.
\end{cases}
\] 

Fix for a while $\alpha>2$, $x\in C$ and $\varepsilon>0$. 
We show that there is $\delta>0$ such that 
\begin{equation}\label{PE.3}
|F(y)-F(x)|\le \varepsilon(\psi(x+\alpha(y-x))-\psi(x)) 
\end{equation}
whenever $0<y-x<\delta$. 
Since $F=0$ on $[1,\infty)$, any $\delta>0$ will do for $x=1$.
So we assume $x<1$. To define $\delta$, start by 
using that 
$\alpha>2$ 
to find $l\in\N$ such that
$\alpha > 2(1 + 5^{-l+1})$. 
Since $\tau_j\to 1/2$ and $\sigma_j\to 0$, there is $m\in\N$ such that
$\alpha \tau_j> 1 + 5^{-l}$ and $\sigma_{j}3^{l+2}<\varepsilon$ for $j\ge m$.
Having done this, we let
$\delta=\min(1-x, 5^{-m-1})$.

We are now ready to prove that \eqref{PE.3} indeed holds for $y\in (x,x+\delta)$.
Given such $y$, find the least $k\ge 1$ for which there is an interval
$(u,v)\in\J_k$ that is contained in $(x,y)$. Clearly, $k>m$.
To finish the argument, we distinguish three cases.

\begin{cs}{$y\in C_{k}$}
Then $x,y$ belong to the same interval from $\I_{k-1}$ since
otherwise $k>1$ and $[x,y]$ contains an interval from $\J_{k-1}$.
Moreover,
$x+\alpha(y-x) = x+\alpha(u-x)+\alpha(y-u) \ge u +\alpha(v-u)$. Since $\alpha>2$
and $v-u=5^{-k}$, we infer that
$[x+\alpha(y-x),x]\supset [u+\alpha(v-u),u]\supset[v,v+5^{-k}]$.
Since $[v,v+5^{-k}]$ belongs to~$\I_k$, we conclude from \ref{F.2} that
\[|F(y)-F(x)|\le \sigma_k 3^{-k+1}= 3\sigma_k (\psi(v+5^{-k})-\psi(v))
\le\tfrac12\varepsilon(\psi(x+\alpha(y-x))-\psi(x)).\]
\end{cs}

In the remaining cases $y\notin C_k$. 
Since $y\in(0,1)$, there is $1\le j\le k$ such that $y\in(w,z)$ for some $(w,z)\in\J_j$.
Since $(x,y)$ contains an interval from $\J_k$ but does not contain any interval
from any $\J_i$ with $i<k$, we have $w\in C_k$,
$x<w$ and $(x,w)$ contains an interval from~$\J_k$. By the previous case,
\begin{equation}\label{pc}
|F(w)-F(x)|\le \tfrac12\varepsilon(\psi(x+\alpha(w-x))-\psi(x)).
\end{equation}

\begin{cs}{$y\notin C_k$ and $y-w\le \tau_j 5^{-j}$} 
Then \ref{F.5} implies that $F$ is constant on $[w,y]$ and so
\eqref{pc} gives
\[|F(y)-F(x)|=|F(w)-F(x)|\le\tfrac12\varepsilon(\psi(x+\alpha(y-x))-\psi(x)).\]
\end{cs}

\begin{cs}{$y\notin C_k$ and $y-w >\tau_j 5^{-j}$} 
We use 
$\delta>y-x\ge y-w\ge \tau_j 5^{-j}\ge 5^{j-1}$ to infer that $j> m$. Hence
$[x,x+\alpha(y-x)]\supset [w,w+\alpha(y-w)]
\supset [z,z+5^{-j-l}]$. Since $[z,z+5^{-j-l}]\in\I_{j+l}$
and $[w,y]$ is contained in an interval from $\I_{j-1}$, \ref{F.2} implies 
\[|F(y)-F(w)|\le\sigma_{j} 3^{-j+1}
=\sigma_j 3^{l+1}(\psi(z+5^{-j})-\psi(z))
\le \tfrac12\varepsilon(\psi(x+\alpha(y-x))-\psi(x)).\]
This and \eqref{pc} give
\[|F(y)-F(x)|\le |F(y)-F(w)|+|F(w)-F(x)|
\le \varepsilon(\psi(x+\alpha(y-x))-\psi(x)).\]
\end{cs}

Having thus verified \eqref{PE.3}, we use it and 
$f(x)=0$ (since $x\in C$) to conclude that
\[\lim_{y \searrow x}{\frac{F(y)-F(x)-f(x)(y-x)}{\varphi(y)-\varphi(x)}}=0.\]
A symmetric argument shows
\[\lim_{y \nearrow x}{\frac{F(y)-F(x)-f(x)(y-x)}{\varphi(y)-\varphi(x)}}=0.\]
Finally, if $x\notin C$, we infer from $|\varphi(y)-\varphi(x)|\ge|y-x|$ and $F'(x)=f(x)$ that
\[\limsup_{y \to x}{\frac{|F(y)-F(x)-f(x)(y-x)|}{|\varphi(y)-\varphi(x)|}}
\le \lim_{y \to x}{\frac{|F(y)-F(x)-f(x)(y-x)|}{|y-x|}}
=0.\]
Hence $F$ is an indefinite \MC{\alpha} integral of $f$, as wanted.

It remains to show that $f$ is not Denjoy-Khintchine integrable. 
Suppose the opposite and let $H$ be its indefinite Denjoy-Khintchine integral.
Notice that, in principle, $H$ may be different 
from $F$. However, on each interval $(a,b)$ from $\J$, 
both $H$ and $F$ are Lebesgue indefinite integrals of $f$
and hence $H-F$ is constant on $[a,b]$. It follows that
$H(b)-H(a)=F(b)-F(a)$, and so \ref{F.4} holds with $F$ replaced by $H$.
This information will suffice for our arguments.

Since $H$ is continuous, $C$ is a union of closed sets on which
it is \AC. By the Baire Category Theorem, there is 
an open interval $I$ meeting $C$ such that $H$  is \AC\ on $I\cap C$. 
Find a component $[u,v]$ of some $C_m$ contained in $I$.
For $k\ge m$ the set $[u,v]\cap C_k$ has $3^{k-m}$ components and so
$[u,v]\cap (C_{k+1}\setminus C_k)$ has $2\cdot 3^{m-k}$ components.
For each such component, say $(u,v)$, we have 
$|H(v)-H(u)|=\sigma_{k+1} 3^{k+1}$ by validity of \ref{F.4} for $H$.

Let $(a_j,b_j)$ be an enumeration of the components of $[u,v]\setminus C$.
Summing first over those $(a_j,b_j)$ that are components of $C_{k+1}\setminus C_k$
and then over $k>m$, we get
$\sum_{j=1}^\infty |H(b_j)-H(a_j)|= \sum_{k=m}^\infty 2\sigma_{k+1} 3^{k+1} 3^{m-k}
=\sum_{k=m}^\infty 2\sigma_{k+1} 3^{m+1}=\infty$. 
This shows that $H$ is not \AC\ on $I\cap C$, and this contradiction
shows that $f$ is not Denjoy-Khintchine integrable.
\end{proof}

\end{document}